\documentclass[10pt,bezier]{article}
\usepackage{amsmath,amssymb,amsfonts,graphicx,caption,subcaption}
\usepackage[colorlinks,linkcolor=red,citecolor=blue]{hyperref}

\newtheorem{prethm}{{\bf Theorem}}[section]
\newenvironment{thm}{\begin{prethm}{\hspace{-0.5
em}{\bf.}}}{\end{prethm}}
\newtheorem{prepro}{{\bf Theorem}}

\newtheorem{precor}[prethm]{{\bf Corollary}}
\newenvironment{cor}{\begin{precor}{\hspace{-0.5
em}{\bf.}}}{\end{precor}}
\newtheorem{preconj}[prethm]{{\bf Conjecture}}

\newtheorem{preremark}[prethm]{{\bf Remark}}
\newenvironment{remark}{\begin{preremark}\em{\hspace{-0.5
em}{\bf.}}}{\end{preremark}}
\newtheorem{prelem}[prethm]{{\bf Lemma}}
\newenvironment{lem}{\begin{prelem}{\hspace{-0.5
em}{\bf.}}}{\end{prelem}}
\newtheorem{preque}[prethm]{{\bf Question}}

\newtheorem{preobserv}[prethm]{{\bf Observation}}
\newenvironment{observ}{\begin{preobserv}{\hspace{-0.5
em}{\bf.}}}{\end{preobserv}}

\newtheorem{predef}[prethm]{{\bf Definition}}

\newtheorem{preproposition}[prethm]{{\bf Proposition}}

\newtheorem{preproof}{{\bf Proof.}}
\newtheorem{preprooff}{{\bf Proof}}

\newenvironment{proof}[1]{\begin{preproof}{\rm
#1}\hfill{$\Box$}}{\end{preproof}}

\newtheorem{preproofs}{{\bf The second proof of }}

\newtheorem{preprooft}{{\bf Third proof of }}

\newtheorem{preproofF}{{\bf Proof of}}

\title{\bf\Large 
Equitable factorizations of edge-connected graphs 
}
\author{{\normalsize{\sc Morteza Hasanvand${}$} }\vspace{3mm}
\\{\footnotesize{${}$\it Department of Mathematical
 Sciences, Sharif
University of Technology, Tehran, Iran}}
{\footnotesize{}}\\{\footnotesize{  $\mathsf{hasanvand@alum.sharif.edu }$ }}}

\date{}

\begin{document}
\maketitle
\begin{abstract}{
In this paper, we show that every $(3k-3)$-edge-connected graph $G$, under a certain condition on whose degrees, can be edge-decomposed into $k$ factors $G_1,\ldots, G_k$ such that for each vertex $v\in V(G_i)$, $|d_{G_i}(v)-d_G(v)/k|< 1$, where $1\le i\le k$. As application, we deduce that every $6$-edge-connected graph $G$ can be edge-decomposed into three factors $G_1$, $G_2$, and $G_3$ such that for each vertex $v\in V(G_i)$, $|d_{G_i}(v)-d_{G}(v)/3|< 1$, unless $G$ has exactly one vertex $z$ with $d_G(z) \stackrel{3}{\not\equiv}0$. Next, we show that every odd-$(3k-2)$-edge-connected graph $G$ can be edge-decomposed into $k$ factors $G_1,\ldots, G_k$ such that for each vertex $v\in V(G_i)$, $d_{G_i}(v)$ and $d_G(v)$ have the same parity and $|d_{G_i}(v)-d_G(v)/k|< 2$, where $k$ is an odd positive integer and $1\le i\le k$. Finally, we give a sufficient edge-connectivity condition for a graph $G$ to have a parity factor $F$ with specified odd-degree vertices such that for each vertex $v$, $| d_{F}(v)-\varepsilon d_G(v)|< 2$, where $\varepsilon $ is a real number with $0< \varepsilon < 1$. 
} {\small
}
\end{abstract}
%
%
%
%
%
%
%
%
%
%
\section{Introduction}
In this article, graphs may have  loops and multiple edges.
Let $G$ be a graph. The vertex set, the edge set, and the maximum degree of $G$ are denoted by 
$V(G)$, $E(G)$, and, $\Delta(G)$, respectively.
We denote by $d_G(v)$ the degree of a vertex $v$ in the graph $G$, whether $G$ is directed or not.
Also the out-degree and in-degree of $v$ in a directed graph $G$ are denoted by $d_G^+(v)$ and $d_G^-(v)$.
An orientation of $G$ is said to be 
{\bf $p$-orientation}, if for each vertex $v$, $d_G^+(v)\stackrel{k}{\equiv}p(v) $,
 where $p:V(G)\rightarrow Z_k$ is a mapping and $Z_k$ is the cyclic group of order $k$.
For any two disjoint vertex sets $A$ and $B$, 
we denote by $d_G(A,B)$ the number of edges with one end in $A$ and the other one in $B$.
Also, we denote by $d_G(A)$ the number of edges of $G$ with exactly one end in $A$.
Note that  $d_G(\{v\})$ denotes the number of non-loop edges incident with $v$,
while $d_G(v)$ denotes the degree of $v$.
We denote by $G[A]$ the induced subgraph of $G$ with the vertex set $A$ containing
precisely those edges of $G$ whose ends lie in $A$.
Let $g$ and $f$ be two integer-valued functions on $V(G)$.
A {\bf parity $(g,f)$-factor} of $G$ refers to a spanning subgraph $F$ such that for each vertex $v$, $g(v)\le d_F(v)\le f(v)$, and
also $d_F(v)$, $g(v)$, and $f(v)$ have the same parity.
An {\bf $f$-parity factor} refers to a spanning subgraph $F$ such that for each vertex $v$, 
$d_F(v)$ and $f(v)$ have the same parity.
A graph $G$ is called {\bf $m$-tree-connected}, if it contains $m$ edge-disjoint spanning trees. 
A graph $G$ is termed {\bf essentially $\lambda$-edge-connected},
 if all edges of any edge cut of size strictly less than $\lambda$ are incident to a vertex.
A graph $G$ is called {\bf odd-$\lambda$-edge-connected}, if $d_G(A)\ge \lambda$, 
for every vertex $A$ with $d_G(A)$ odd.
Note that $2\lambda$-edge-connected graphs are
odd-$(2\lambda+1)$-edge-connected.
A factorization $G_1,\ldots, G_k$ of $G$ is called 
{\bf equitable factorization}, if for each vertex $v$, $|d_{G_i}(v)-d_G(v)/k|<1$, where $1\le i\le k$.
Throughout this article, all variables $k$ are positive and integer.

In 1971 de Werra observed that bipartite graphs admit equitable factorizations.
\begin{thm}{\rm (\cite{Werra})}\label{thm:Werra}
{Every  bipartite graph  $G$ can be edge-decomposed into $k$ factors $G_1,\ldots, G_k$ such that for each $v\in V(G_i)$, 
 $$ | d_{G_i}(v)-d_G(v)/k|< 1 .$$
}\end{thm}

In 1994 Hilton and de Werra developed Theorem~\ref{thm:Werra} to simple graphs by replacing the following weaker version.
In 2008 Hilton~\cite{HILTON2008645} conjectured that for those simple graphs for which whose vertices with degree divisible by $k$ form an induced forest, the upper bound can be replaced by that of Theorem~\ref{thm:Werra}. 
Later, this conjecture was confirmed by Zhang and Liu (2011)~\cite{Zhang-Liu}.
\begin{thm}{\rm (\cite{Hilton-Werra})}\label{thm:Hilton-Werra}
{Every simple graph $G$ can be edge-decomposed into $k$ factors $G_1,\ldots, G_k$ such that for each $v\in V(G_i)$, 
 $$ | d_{G_i}(v)-d_G(v)/k|\le 1 .$$
}\end{thm}

Recently, Thomassen (2019) 
made another attempt for developing Theorem~\ref{thm:Werra} to highly edge-connected regular graphs 
and concluded the following result.
An intresting consequence of it says that every $6$-edge-connected $3r$-regular graph can be edge-decomposed into three $r$-regular factors.
\begin{thm}{\rm (\cite{Thomassen})}\label{thm:Thomassen:factorization}
{Let $k$ and $r$ be two odd positive integers. If $G$ 
is an odd-$(3k-2)$-edge-connected $kr$-regular graph, then it can be edge-decomposed into $r$-regular factors.
}\end{thm}

In this paper, we generalize Thomassen's result in several ways
 together with developing Theorem~\ref{thm:Werra} to highly edge-connected graphs as the next theorem. 
In addition, we provide a common version for Theorem~\ref{thm:Thomassen:factorization} and whose even-regular  version
by replacing another concept of edge-connectivity.
\begin{thm}\label{Intro:thm:Z}
{Let $G$ be a $(3k-3)$-edge-connected graph, where $k$ is a positive integer. 
If there is a vertex set $Z$ with $|E(G)|\stackrel{k}{\equiv}\sum_{v\in Z}d_G(v)$,
 then $G$ can be edge-decomposed into $k$ factors $G_1,\ldots, G_k$ such that for each $v\in V(G_i)$, 
 $$|d_{G_i}(v)-d_G(v)/k|< 1 .$$
}\end{thm}

In 1982 Hilton constructed the following parity version of Theorem~\ref{thm:Werra} 
on the existence of even factorizations with bounded degrees of even graphs.

\begin{thm}{\rm (\cite{Hilton1982})}\label{thm:Hilton}
{Every  graph $G$ with even degrees can be edge-decomposed into $k$ even factors $G_1,\ldots, G_k$ such that for each vertex $v$, 
 $|d_{G_i}(v)-d_{G}(v)/k|<2.$
}\end{thm}

In this paper, we provide two extensions of Hilton's result in highly odd-edge-connected graphs and highly edge-connected even graphs by proving the following results. 
An interesting consequence of them says that every $(3k-3)$-edge-connected graph of even order whose degrees lie in the set $\{3k, 3k+2,...,5k\}$ can be edge-decomposed into $k$ factors whose degrees lie in the set $\{3,5\}$.
Note that the next theorem can also be considered as an extension of Theorem~\ref{thm:Thomassen:factorization}.
\begin{thm}\label{thm:parity-factorizations:odd-k}
{Let $G$ be a graph and let $k$ be an odd positive integer.
If $G$ is odd-$(3k-2)$-edge-connected, then it can be edge-decomposed into $k$ factors $G_1,\ldots, G_k$ such that for each $v\in V(G_i)$, 
 $ d_{G_i}(v)\stackrel{2}{\equiv} d_G(v)$
and
 $$|d_{G_i}(v)-d_{G}(v)/k|<2.$$
}\end{thm}
\begin{thm}\label{thm:parity-factorizations:even-k}
{Let $G$ be a graph with even degrees, let $f:V(G)\rightarrow Z_2$ be a mapping with
$ \sum_{v\in V(G)}f(v)\stackrel{2}{\equiv}0$, and let $k$ be an even positive integer.
If for every vertex set $X$ 
with $ \sum_{v\in X}f(v)$ odd,
$d_G(X)\ge 3k-2$, then $G$ can be edge-decomposed into $f$-parity factors $G_1,\ldots, G_k$ such that for each $v\in V(G_i)$, 
 $$|d_{G_i}(v)-d_{G}(v)/k|<2.$$
}\end{thm}

In the remainder of this paper, we turn our attention to the existence of a parity factor 
whose degrees are close to $\varepsilon$ of the corresponding degrees in the main graph, and
we come up with the following results on even graphs.
\begin{thm}\label{thm:intro:even-even}
{Let $\varepsilon $ be a real number with $0\le  \varepsilon\le 1$. If $G$ is a  graph with even degrees,
then it admits an even factor $F$ such that for each vertex $v$, $$ | d_{F}(v)-\varepsilon d_G(v)| < 2.$$
}\end{thm}
\begin{thm}\label{thm:intro:even-factorization}
{Let $\varepsilon_1,\ldots,\varepsilon_k$ be $k$ positive real numbers with $\varepsilon_1+\cdots +\varepsilon_k=1$.
If $G$ is a graph with even degrees, then it can be edge-decomposed into $k$ even factors $G_1,\ldots, G_k$ such that for each $v\in V(G_i)$,
 $|d_{G_i}(v)-\varepsilon_i d_G(v)|<6.$
}\end{thm}
%
%
%
%
%
%
%
%
%
\section{Preliminary results on directed graphs}
The following theorem provides a relationship between orientations and equitable factorizations of graphs which is motivated by Theorem 2 in~\cite{Thomassen}. This result plays an important role in the next section.
\begin{thm}\label{thm:factorization:main:directed}
{Every directed graph $G$ can be edge-decomposed into $k$ factors $G_1,\ldots, G_k$
 such that  for each $v\in V(G_i)$ that $1\le i\le k$,
$$ | d^+_{G_i}(v)-d^+_G(v)/k|< 1 \text{ and } | d^-_{G_i}(v)- d^-_G(v)/k|< 1.$$
In particular, $ | d_{G_i}(v)- d_G(v)/k|< 1$, when $ d^+_G(v)$ or $d^-_G(v)$ is divisible by $k$.
Furthermore, one can impose at least one of the following conditions arbitrary:
\begin{enumerate}{
\item [$(i)$] $|\,|E(G_i)|-|E(G)|/k\,|<1$, where $1\le i\le k$.
\item [$(ii)$] If $u\in V(G)$ and $d_G^+(u)\stackrel{k}{\equiv}d_G^-(u)$, then
$d_{G_i}(u)\stackrel{2}{\equiv}(d_G^+(u)-d_G^-(u))/k$, where $1\le i\le k$.
}\end{enumerate}
}\end{thm}
\begin{proof}
{The proof presented here is inspired by the proof of Theorem 2 in~\cite{Thomassen}. 
First split every vertex $v$ into two vertices $v^+$ and $v^-$ such that
 $v^+$ whose incident edges were directed away from $v$ in $G$ and 
$v^-$ whose incident edges were directed toward $v$ in $G$.
In this construction,  every loop in $G$ incident with $v$ is transformed into an edge between $v^+$ and $v^-$.
Now, split every vertex in the new graph into vertices with degrees divisible by $k$, 
except possibly one vertex with degree less than $k$. Call the resulting (loopless) bipartite graph $G_0$.
Since $G_0$ has maximum degree at most $k$,
 it admits a proper edge-coloring with at most $k$ colors $c_1\ldots, c_k$
(to prove this, it is enough to consider a $k$-regular  supergraph of it and apply  K{\"o}nig Theorem~\cite{Konig}).
Let $G_i$ be the factor of $G$ consisting of the edges of $G$ corresponding to the edges with color $c_i$ in $G_0$.
Note that for every vertex $v$, at least $\lfloor d^+_G(v)/k\rfloor$ (resp. $\lfloor d^-_G(v)/k\rfloor$) edges with color $c_i$ are incident with the vertex $v^+$ (resp. $v^-$) in $G_0$.
Moreover, at most $\lceil d^+_G(v)/k\rceil$ (resp. $\lceil d^-_G(v)/k\rceil$)  edges with color $c_i$ are incident with the vertex $v^+$ (resp. $v^-$) in $G_0$. This completes the proof of the first part.
 
First we are going to prove item (i).
Let us consider such a proper edge-coloring with the minimum $\sum_{1\le i\le k}|m_i-m/k|$, where $m_i=|E(G_i)|$ and  $m=|E(G)|$.
We claim that $|m_i-m/k|<1$ for every color $c_i$.
Suppose, to the contrary, that there is a color $c_i$ with $|m_i-m/k|\ge 1$. 
We may assume that $m_i\ge m/k+1$; as the proof of the case $m_i\le m/k-1$ is similar.
Thus there is another color $c_j$ with $c_j<m/k$ so that $m_i> m_j+1$.
Let $H$ be the factor of $G_0$ consisting of the edges that are colored with $c_i$ or $c_j$.
According to the proper edge coloring of $G_0$, the graph $H$ must be the union of some paths and cycles.
 Since all cycles have even size and $m_i> m_j$, it is easy to check that there is a path $P$ 
such that whose edges are colored alternatively by $c_i$ and $c_j$, and also whose end edges are colored with $c_i$. 
Now, it is enough to exchange the colors of the edges of $P$ to find another proper edge-coloring satisfying 
   $m'_i=m_i-1$ and $m'_j=m_j+1$ where $m'_i$ and $m'_j$ are the number of edges having  these new colors.
Since $|m'_j-m/k|+|m'_i-m/k|< (|m_j-m/k|+1)+(|m_i-m/k|-1)$, one can easily derive at a contradiction.
Therefore, $G_1,\ldots, G_k$ are the desired factors we are looking for.

Now, we are going to prove item (ii).
Let us add some artificial edges to the graph $G_0$.
For every vertex $u$ with $d_G^+(u)\stackrel{k}{\equiv}d_G^-(u)\stackrel{k}{\not\equiv}0$,
 denote by $u_0^+$ and $u_0^-$  the two vertices in $G_0$ having the same degree strictly less than $k$ obtained from splitting of $u^+$ and $u^-$.
Add some new artificial edges between these two vertices to construct these vertices with degree $k$.
Apply this method for all such vertices $u$ described above.
Consider a  proper edge-coloring with at most $k$ colors $c_1\ldots, c_k$ for the new resulting bipartite graph $G'_0$.
Similarly, we define $G_i$ to be the factor of $G$ consisting of the edges of $G$ corresponding to the edges with color $c_i$ in $G_0$.
Note that for every vertex $u$ with $d_G^+(u)\stackrel{k}{\equiv}d_G^-(u)\stackrel{k}{\not\equiv}0$, 
either one edge between $u_0^+$ and $u_0^-$ in $G'_0$ (may be an artificial edge) is colored with $c_i$ 
or two edges incident with $u_0^+$ and $u_0^-$ in $G_0$ are colored with $c_i$.
Thus for every vertex $u$ with $d_G^+(u)\stackrel{k}{\equiv}d_G^-(u)$, 
we must have $d^-_{G_i}(u)- d^-_G(u)/k = d^+_{G_i}(u)-d^+_G(u)/k $ and so 
$d_{G_i}(u)\stackrel{2}{\equiv}d^+_{G_i}(u) -d^-_{G_i}(u)=(d_G^+(u)-d_G^-(u))/k$.
Therefore, $G_1,\ldots, G_k$ are again the desired factors we are looking for.
}\end{proof}
\begin{cor}{\rm (\cite{Anstee})}
{Every graph $G$ can be edge-decomposed into $k$ factors $G_1,\ldots, G_k$ such that  for each factor $G_i$, 
$|\,|E(G_i)|-|E(G)|/k\, |<1$, and for each $v\in V(G_i)$,
 $$
 \lfloor\frac{d_G(v)}{2k}\rfloor+\lfloor\frac{d_G(v)+1}{2k}\rfloor  \le d_{G_i}(v)\le\lceil\frac{d_G(v)}{2k}\rceil+\lceil\frac{d_G(v)-1}{2k}\rceil .$$
}\end{cor}
\begin{proof}
{First, consider an orientation for $G$ such that for each vertex $v$, 
 $|d^+_G(v)-d^-_G(v)|\le 1$.
Next, apply Theorem~\ref{thm:factorization:main:directed} (i) to this directed graph.
}\end{proof}
%
%
%
%
%
%
%
%
%
%
%
%
\section{Factorizations of edge-connected graphs}
In this section, we are going to give a sufficient condition for the existence of equitable factorizations in highly edge-connected graphs. For this purpose,  let us recall the following lemma from~\cite{ModuloBounded, MR3096333}.
\begin{lem}{\rm (\cite{ModuloBounded, MR3096333})}\label{lem:modulo:2k-2:3k-3}
{Let $G$ be a graph, let $k$ be an integer, $k\ge 2$, and let $p:V(G)\rightarrow Z_k$ be a mapping with
$|E(G)| \stackrel{k}{\equiv} \sum_{v\in V(G)}p(v)$.
If $G$ is $(3k-3)$-edge-connected or $(2k-2)$-tree-connected, 
then it admits a $p$-orientation modulo $k$.
Furthermore, the result holds for loopless $(2k-1)$-edge-connected essentially $(3k-3)$-edge-connected graphs provided that $p(v)=0$ or $p(v)\stackrel{k}{\equiv}d_G(v)$ for each vertex $v$.
}\end{lem}
Before stating the main result, we begin with the following theorem that only needs replacing a weaker condition on  a single arbitrary vertex. 
\begin{thm}\label{thm:factorization:main}
{Let $G$ be a graph with $z\in V(G)$ and let $k$ be a positive integer. 
If $G$ is $(3k-3)$-edge-connected or $(2k-2)$-tree-connected,
 then it can be edge-decomposed into $k$ factors $G_1,\ldots, G_k$ satisfying $||E(G_i)|-|E(G)|/k|<1$   such that for each $v\in V(G_i)$,
 $$ | d_{G_i}(v)-d_G(v)/k| < \begin{cases}
2	&\text{when $v= z$};\\
1,	&\text{otherwise}.
\end {cases}$$
Furthermore, for the vertex $z$, we can also have $| d_{G_i}(z)-d_{G_j}(z)|\le 2$.
}\end{thm}
\begin{proof}
{According to Lemma~\ref{lem:modulo:2k-2:3k-3}, the graph $G$ admits an orientation such that for each $v\in V(G)\setminus z $, $d^+_G(v)\stackrel{k}{\equiv}0$, and $d^+_G(z)\stackrel{k}{\equiv}|E(G)|$.
Now, it is enough to apply Theorem~\ref{thm:factorization:main:directed} (i).
}\end{proof}
Now, we are reedy to state the main result of this section  which is a strengthened version of Theorem~\ref{Intro:thm:Z}.
\begin{thm}\label{thm:factorization:Z}
{Let $G$ be a graph and let $k$ be a positive integer. 
Assume that $G$ is $(3k-3)$-edge-connected or $(2k-2)$-tree-connected.
If there is a vertex set $Z$ with $|E(G)|\stackrel{k}{\equiv}\sum_{v\in Z}d_G(v)$,
 then $G$ can be edge-decomposed into $k$ factors $G_1,\ldots, G_k$ satisfying $||E(G_i)|-|E(G)|/k|<1$ such that for each $v\in V(G_i)$, 
 $$|d_{G_i}(v)-d_G(v)/k|< 1 .$$
Furthermore, the result holds for loopless $(2k-1)$-edge-connected essentially $(3k-3)$-edge-connected graphs.
}\end{thm}
\begin{proof}
{For each $v\in Z$, define $p(v)=d_G(v)$, and for each $v\in V(G)\setminus Z$, define $p(v)=0$.
By the assumption, we have $|E(G)|\stackrel{k}{\equiv}\sum_{v\in V(G)}p(v)$.
Thus by Lemma~\ref{lem:modulo:2k-2:3k-3}, the graph $G$ admits a $p$-orientation modulo $k$
 so that for each $v\in Z$, $d^-_G(v)$ is visible by $k$ and for each $v\in V(G)\setminus Z$, $d^+_G(v)$ is divisible by $k$.
Hence the assertion follows from Theorem~\ref{thm:factorization:main:directed} (i).
}\end{proof}
Graphs with size divisible by $k$ are natural candidates for graphs satisfying the assumptions of Theorem~\ref{thm:factorization:Z}. We examine them to deduce the following corollary.
\begin{cor}\label{cor:size:divisible-by-k}
{Let $G$ be a graph of size divisible by $k$.
If $G$ is $(3k-3)$-edge-connected or $(2k-2)$-tree-connected,
 then it can be edge-decomposed into $k$ factors $G_1,\ldots, G_k$ with the same size such that for each $v\in V(G_i)$, 
 $$|d_{G_i}(v)-d_G(v)/k|< 1 .$$
}\end{cor}
\begin{proof}
{Apply Theorem~\ref{thm:factorization:Z} with $Z=\emptyset$.  Note that $|E(G)|$ is divisible by $k$.
}\end{proof}
The next corollary gives a criterion for  the existence of a $3$-equitable factorization in edge-connected graphs.
\begin{cor}
{Let $G$ be a $6$-edge-connected graph.
Then $G$ has not exactly one vertex $z$ with $d_G(z) \stackrel{3}{\not\equiv}0$ if and only if it can be edge-decomposed into three factors $G_1$, $G_2$, and $G_3$ such that for each $v\in V(G_i)$, 
 $$|d_{G_i}(v)-d_G(v)/3|< 1 .$$
Furthermore, the result holds for loopless $5$-edge-connected essentially $6$-edge-connected graphs.
}\end{cor}
\begin{proof}
{Set $Z_0$ to be the set of all vertices of $G$ such that whose degrees are not divisible by $3$.
First assume that $|Z_0|\neq 1$. If $|Z_0|=0$, then $2|E(G)|$ must be divisible by $3$ and so $|E(G)|$ must be divisible by $3$.
Therefore, it is easy to check that there is a subset $Z$ of $Z_0$ such that 
$|E(G)|\stackrel{3}{\equiv}\sum_{v\in Z}d_G(v)$ whether $|Z_0|\ge 2$ or not. 
 Hence  by  Theorem~\ref{thm:factorization:Z}, the graph $G$ has the desired factorization.
Now, assume  $Z_0=\{z\}$.
Suppose, to the contrary, that $G$ can be edge-decomposed into three factors $G_1$, $G_2$, and $G_3$ such that for each $v\in V(G_i)$, $|d_{G_i}(v)-d_G(v)/3|< 1$.
We may therefore assume that 
$d_{G_1}(z)=\lfloor d_G(z)/3\rfloor$ and 
$d_{G_2}(z)=\lceil d_G(z)/3\rceil$. On the other hand, for all vertices $v\in V(G)\setminus \{z\}$, $d_{G_1}(v)=d_{G_2}(v)=d_G(v)/3$ which implies that $d_{G_1}(z)$ and $d_{G_2}(z)$ have the same parity, because of the handshaking lemma. This is contradiction, as desired. 
}\end{proof}
As we observed above, graphs with a number of vertices whose degrees are not divisible by $k$ are other natural candidates 
for graphs satisfying the assumptions of Theorem~\ref{thm:factorization:Z}. By employing the following lemma, we examine a special case  of them to imply the next corollary.
\begin{lem}{\rm (Chowla \cite{Chowla})}\label{lem:Chowla}
{Let $k$ be an integer number with $k\ge 2$ and let $x_1, \ldots, x_{k-1}\in Z_k\setminus \{0\}$ be $k-1$ integer numbers coprime with $k$ (not necessarily distinct).
If $m\in Z_k$, then there is a subset $Z\subseteq \{1,\ldots, k-1\}$ such that
 $m\stackrel{k}{\equiv}\sum_{i\in Z}x_i$.
}\end{lem}
\begin{cor}\label{cor:prime}
{Let $G$ be a $(3k-3)$-edge-connected graph where $k$ is a prime number.
If $G$ contains at least $k-1$ vertices whose degree are coprime with $k$, then it can be edge-decomposed into $k$ factors $G_1,\ldots, G_k$ satisfying $||E(G_i)|-|E(G)|/k|<1$ such that for each $v\in V(G_i)$, 
 $|d_{G_i}(v)-d_G(v)/k|< 1.$
}\end{cor}
\begin{proof}
{Let $Z_0$ be the set of all vertices of $G$ such that whose degrees are coprime with $k$.
Since $|Z_0|\ge k-1$,  there is a subset $Z$ of $Z_0$ such that 
$|E(G)|\stackrel{k}{\equiv}\sum_{v\in Z}d_G(v)$ according to Lemma~\ref{lem:Chowla}.
Hence the assertion follows from Theorem~\ref{thm:factorization:Z} immediately.
}\end{proof}
It is perhaps surprising that the lower bound of $k-1$ in the above-mentioned corollary is best possible according to the following observation. It remains to decide whether the statement of Corollary~\ref{cor:prime} holds
 if $G$ contains at least $k-1$ vertices whose degree are not divisible by $k$.
The statement is obviously true for all prime numbers $k$.
\begin{observ}
{For every integer $k$ with $k\ge 2$, there are infinitely many highly edge-connected graphs $G$ having at least $k-2$ 
vertices whose degrees are  coprime with $k$, while $G$  cannot be edge-decomposed into $k$ factors $G_1,\ldots, G_k$  such that for each $v\in V(G_i)$, $|d_{G_i}(v)-d_G(v)/k|< 1$.
}\end{observ}
\begin{proof}
{Let $r$ be an arbitrary odd positive integer. 
Choose a $kr$-edge-connected graph $G$ of odd order
 in which whose degrees are $rk$ except for $k-2$ vertices having degree $kr+1$
(for $k$ even, one can consider a $kr$-edge-connected  $kr$-regular
 large graph of odd order and insert a new perfect matching of size $k/2-1$ to it).
If $G$ has the desired factorization, then according to the vertex degree, 
there must be a factor $G_i$ such that for all vertices $v$, $d_{G_i}(v)=\lfloor d_G(v)/k\rfloor$.
This implies that $G$  an $r$-regular factor which is not possible, because of the handshaking lemma. 
Hence the assertion holds.
}\end{proof}
%
%
%
%
%
%
%
\section{Parity factorizations of odd-edge-connected graphs}
In this section, we are going to develop each of Theorems~\ref{thm:Thomassen:factorization} and~\ref{thm:Hilton} in two ways
 based on Theorem~\ref{thm:factorization:main:directed}. For this purpose, we first need the following lemma which improves the edge-connectivity needed in Lemma~\ref{lem:modulo:2k-2:3k-3} for a special case.
\begin{lem}{\rm (\cite{MR3096333})}\label{lem:balanced-modulok:odd}
{Let $k$ be an odd positive integer. 
If $G$ is an odd-$(3k-2)$-edge-connected graph, 
then it admits an orientation such that for each vertex $v$, $d_G^+(v)\stackrel{k}{\equiv}d_G^-(v)$.
}\end{lem}
The following theorem generalizes Theorem~\ref{thm:Hilton} to non-Eulerian graphs for the case that $k$ is odd.
This result can also be considered as an improvement of a result due to Shu, Zhang, and Zhang (2012) \cite{Shu-Zhang-Zhang}
who proved this result  for odd-$(2k-1)$-edge-connected graphs without considering the restriction on vertex degrees.
\begin{thm}\label{thm:parity-factorizations:odd-k}
{Let $k$ be an odd positive integer.
If $G$ is an odd-$(3k-2)$-edge-connected or $(2k-2)$-tree-connected graph, then it can be edge-decomposed into $k$ factors $G_1,\ldots, G_k$ such that for each $v\in V(G_i)$, 
 $ d_{G_i}(v)\stackrel{2}{\equiv} d_G(v)$
and
 $$|d_{G_i}(v)-d_{G}(v)/k|<2.$$
}\end{thm}
\begin{proof}
{According to Lemmas~\ref{lem:modulo:2k-2:3k-3} and~\ref{lem:balanced-modulok:odd}, 
the edge-connectivity condition implies the graph $G$ admits an orientation such that
for each vertex $v$, $d_G^+(v)\stackrel{k}{\equiv}d_G^-(v)$.
By Theorem~\ref{thm:factorization:main:directed} (ii), the graph $G$ can be edge-decomposed into $k$ factors 
 $G_1,\ldots, G_k$ such that  for each $v\in V(G_i)$,
$| d^-_{G_i}(v)- d^-_G(v)/k|< 1$ and $| d^+_{G_i}(v)-d^+_G(v)/k|< 1$, and also
$d_{G_i}(v)\stackrel{2}{\equiv}(d_G^+(v)-d_G^-(v))/k$.
Therefore, $|d_{G_i}(v)-d_{G}(v)/k|\le |d^+_{G_i}(v)-d^+_{G}(v)/k|+|d^-_{G_i}(v)-d^-_{G}(v)/k|<2$,
and all $d_{G_i}(v)$ have the same parity for which $1\le i\le k$.
Since $d_G(v)=\sum_{1\le j\le k}d_{G_j}(v)$ and $k$ is odd, $d_{G_i}(v)$ and $d_{G}(v)$ must have the same parity. Hence the theorem holds.
}\end{proof}
The edge-connectivity needed in Lemma~\ref{lem:modulo:2k-2:3k-3} can also be improved for the following special case.
We are going to apply it to deduce the next result on Eulerian graphs.
\begin{lem}{\rm (\cite{ModuloBounded})}\label{lem:balanced-modulok:even}
{Let $k$ be an even positive integer, let $G$ be an Eulerian graph, and let $Q\subseteq V(G)$ with $|Q|$ even.
If $d_G(X)\ge 3k-2$ for every $X\subseteq V(G)$ with $|X\cap Q|$ odd,
then $G$ admits an orientation such that for each $v\in V(G)\setminus Q$, $d_G^+(v)=d_G^-(v)$, and 
for each $v\in Q$, $|d_G^+(v)-d_G^-(v)|=k$.
}\end{lem}
The next theorem strengthens Theorem~\ref{thm:Hilton} by imposing a new parity restriction on degrees of the desire factors for graphs with higher edge-connectivity.
\begin{thm}\label{thm:parity-factorizations:even-k}
{Let $k$ be an even positive integer, let $G$ be an Eulerian graph, and let $f:V(G)\rightarrow Z_2$ be a mapping with
$ \sum_{v\in V(G)}f(v)$ even.
If $d_G(X)\ge 3k-2$ for every vertex set $X$ 
with $ \sum_{v\in X}f(v)$ odd, or $G$ is $(2k-2)$-tree-connected, then $G$ can be edge-decomposed into $f$-parity factors $G_1,\ldots, G_k$ such that for each $v\in V(G_i)$, 
 $$|d_{G_i}(v)-d_{G}(v)/k|<2.$$
}\end{thm}
\begin{proof}
{According to Lemmas~\ref{lem:modulo:2k-2:3k-3} and~\ref{lem:balanced-modulok:even}, 
the edge-connectivity condition implies the graph $G$ admits an orientation
 such that for each vertex $v$, $d_G^+(v)=d_G^-(v)$ when $f(v)$ is even, and $|d_G^+(v) -d_G^-(v)|=k$ when $f(v)$ is odd.
Thus by Theorem~\ref{thm:factorization:main:directed} (ii), the graph $G$ can be edge-decomposed into $k$ factors 
 $G_1,\ldots, G_k$ such that  for each $v\in V(G_i)$,
$| d^-_{G_i}(v)- d^-_G(v)/k|< 1$ and $| d^+_{G_i}(v)-d^+_G(v)/k|< 1$, and also
$d_{G_i}(v)\stackrel{2}{\equiv}(d_G^+(v)-d_G^-(v))/k$.
Therefore, $|d_{G_i}(v)-d_{G}(v)/k|\le |d^+_{G_i}(v)-d^+_{G}(v)/k|+|d^-_{G_i}(v)-d^-_{G}(v)/k|<2$
and $d_{G_i}(v)\stackrel{2}{\equiv}f(v)$.
Hence the theorem holds.
}\end{proof}
An attractive application of Theorems~\ref{thm:parity-factorizations:odd-k} and~\ref{thm:parity-factorizations:even-k} is given in the following corollary.
\begin{cor}
{Let $k$ and $r$ be two positive integers. 
Let $G$ be a graph whose degrees lie in the set $\{rk, rk+2,...,rk+2k\}$ for which $r|V(G)|$ is even.
If $G$ is $(3k-3)$-edge-connected or $(2k-2)$-tree-connected, then it can be edge-decomposed into $k$ $\{r,r+2\}$-factors.
}\end{cor}
\begin{proof}
{If $k$ is odd, then by Theorem~\ref{thm:parity-factorizations:odd-k}, the graph $G$ can be edge-decomposed into $k$ factors $G_1,\ldots, G_k$ such that for each $v\in V(G_i)$, 
 $ d_{G_i}(v)\stackrel{2}{\equiv} d_G(v)\stackrel{2}{\equiv}r$ and 
 $r-2\le d_{G}(v)/k -2<d_{G_i}(v)<d_{G}(v)/k+2\le r+4.$
If $k$ is even, then by applying Theorem~\ref{thm:parity-factorizations:even-k} with $f(v)=r$, these factors can similarly be found
such that for each $v\in V(G_i)$, 
 $ d_{G_i}(v)\stackrel{2}{\equiv} r$ and 
 $r-2<d_{G_i}(v)< r+4.$
This completes the proof.
}\end{proof}
%
%
%
%
%
%
%
%
%
%
%
\subsection{Graphs with degrees divisible by $k$: regular factorizations}
In this subsection, we restrict out attention to graphs with degrees divisible by $k$ and derive some results based on the following reformulation of Theorems~\ref{thm:parity-factorizations:odd-k} and~\ref{thm:parity-factorizations:even-k} on this family of graphs.
\begin{thm}\label{thm:divisiblebyk}
{Let $k$ be a positive integer and let $G$ be a graph with size and degrees divisible by $k$.
Take $Q$ to be the set of all vertices $v$ with $d_G(v)/k$ odd. 
If for every vertex X with $|X\cap Q|$ odd,
$$d_G(X)\ge 3k-2,$$
then $G$ can be edge-decomposed into $k$ factors $G_1,\ldots, G_k$ such that for each $v\in V(G_i)$, 
 $d_{G_i}(v)=d_G(v)/k$.
}\end{thm}
\begin{proof}
{To show a a directive proof, we first consider an orientation for $G$ such that out-degree of each vertex is divisible by $k$
using Lemmas~\ref{lem:balanced-modulok:odd} and~\ref{lem:balanced-modulok:even}.
Next, it is enough to apply Theorem~\ref{thm:factorization:main:directed} (i).
}\end{proof}
The following corollary partially confirms Conjecture 2 in~\cite{Thomassen} by giving a supplement for Theorem~\ref{thm:Thomassen:factorization}. 
\begin{cor}\label{cor:r-regular:factorization}
{Let $r$ be an odd positive integer. If $G$ is a $kr$-regular graph $G$ of even order satisfying $d_G(A)\ge 3k-2$, for every vertex set $A$ with $|A|$ odd, then it can be edge-decomposed into $r$-factors.
}\end{cor}
\begin{proof}
{Since $G$ has even order, its size must be divisible by $k$. Note also that for each vertex $v$, $d_G(v)/k$ is odd. Thus the assertion follows from Theorem~\ref{thm:divisiblebyk} immediately, 
}\end{proof}
The following corollary gives a supplement for Theorem 3 in \cite{Thomassen}.
\begin{cor}\label{cor:r1rm:factorizations}
{Let $r$ be a positive integer
and let $r_1,\ldots, r_m$ be $m$ positive integers satisfying
 $r=r_1+\cdots+r_m$ and $r_i\ge r/k-1\ge 2$ in which $k$ is positive divisor of $r$ with $r/k$ odd.
If $G$ is an $r$-regular graph of even order and for every vertex $X$ with $|X|$ odd,
$$d_G(X)\ge 3k-2,$$
then it can be edge-decomposed into factors $G_1,\ldots, G_m$ such that every graph $G_i$ is $r_i$-regular.
}\end{cor}
\begin{proof}
{Apply Corollary~\ref{cor:r-regular:factorization} along with a similar argument stated in the proof of Theorem 3 in \cite{Thomassen}.
}\end{proof}
The following corollary is a counterpart of Corollary~\ref{cor:r1rm:factorizations} and replaces a weaker edge-connectivity condition compared to Theorem 4 in \cite{Thomassen}. 
The proof technique shows a worthwhile application of this kind of odd-edge-connectivity for working with supergraphs.
\begin{cor}
{Let $r$ be a positive integer
and let $r_1,\ldots, r_m$ be $m$ positive integers satisfying
 $r=r_1+\cdots+r_m$ and $r_i\ge r/k-1\ge 2$ in which $k$ is positive divisor of $r$ with $r/k$ odd. 
If $G$ is a graph with $r|V(G)|$ even satisfying $\Delta(G)\le r$ and for every vertex X with $|X|$ odd,
$$d_G(X)\ge 3k-3,$$
then it can be edge-decomposed into factors $G_1,\ldots, G_m$ satisfying $\Delta(G_i)\le r_i$ for each $i$ with $1\le i \le m$. 
}\end{cor}
\begin{proof}
{Add some edges to $G$, as long as possible, such that the resulting graph $G'$ has maximum degree at most $r$.
Since adding loops is possible and $r|V(G)|$ is even, the graph $G'$ must be $r$-regular.
Obviously, for every vertex set $X$ with $|X|$ odd, we still have $d_{G'}(X)\ge 3k-3$.
In addition, since $d_{G'}(X)$ and $k$ have the same parity, it implies that $d_{G'}(X)\ge 3k-2$. 
Thus by Corollary~\ref{cor:r1rm:factorizations}, the graph $G'$ can be edge-decomposed into factors $G'_1,\ldots, G'_m$ such that every graph $G'_i$ is $r_i$-regular. Now, it is enough to induce this factorization for $G$ to complete the proof.
}\end{proof}
%
%
%
%
%
%
%
%
%
\section{A sufficient edge-connectivity condition for the existence of a parity factor}
\label{sec:epsilon-parity-factor}
In 1956 Hoffman made the following theorem 
on the the existence factors 
in bipartite graphs
 whose degrees are close to $\varepsilon$ of the corresponding degrees in the main graph. 
\begin{thm}{\rm (\cite{Hoffman})}\label{thm:Hoffman}
{Let $\varepsilon $ be a real number with $0\le \varepsilon\le 1$. If $G$ is a bipartite graph,
then it has a factor $F$ such that for each vertex $v$, 
 $ | d_{F}(v)-\varepsilon d_G(v)|< 1 .$
}\end{thm}
In 1983  Kano and Saito generlized Theorem~\ref{thm:Hoffman} to general graphs as the next version.
\begin{thm}{\rm (\cite{Kano-Saito})}\label{thm:Kano-Saito}
{Let $\varepsilon $ be a real number with $0\le \varepsilon\le 1$. 
If $G$ is a graph, then it has a factor $F$ such that for each vertex $v$, 
 $| d_{F}(v)-\varepsilon d_G(v)|\le 1 .$
}\end{thm}
In 2007 Correa and Goemans~\cite{Correa-Goemans} formulated the following factorization version of Theorem~\ref{thm:Hoffman}  for bipartite graphs. Later,  Correa and Matamala  (2008)~\cite{Correa-Matamala} remarked that 
it is possible to generalize their result to general graphs
by replacing Theorem~\ref{thm:Kano-Saito} in their proof, and Feige and Singh (2008)~\cite{Feige-Singh} 
 introduced an interesting alternative proof for this theorem
 using linear algebraic techniques.
\begin{thm}{\rm (\cite{Correa-Matamala, Correa-Goemans})}\label{thm:Correa-Goemans}
{Let $\varepsilon_1,\ldots,\varepsilon_k$ be $k$ nonnegative real numbers with $\varepsilon_1+\cdots +\varepsilon_k=1$.
If $G$ is a graph, then it can be edge-decomposed into $k$ factors $G_1,\ldots, G_k$ such that for each $v\in V(G_i)$,
 $$|d_{G_i}(v)-\varepsilon_i d_G(v)|<3.$$
}\end{thm}
In this section, we shall prove the parity versions of Theorems~\ref{thm:Kano-Saito} and~\ref{thm:Correa-Goemans}
 which were mentioned in the introduction as Theorems~\ref{thm:intro:even-even} and~\ref{thm:intro:even-factorization}.
Besides them, 
 we also generalize a recent result in~\cite{ModuloBounded}.
Our proofs are based on the following well-known lemma due to Lov{\'a}sz (1972).
\begin{lem}{\rm (\cite{Lovasz-1972})}\label{lem:Lovasz:parity}
{Let $G$ be a graph and let $g$ and $f$ be two integer-valued functions on $V(G)$ satisfying 
$g(v)\le f(v)$ and $g(v)\stackrel{2}{\equiv}f(v)$ for each $v\in V(G)$. 
Then $G$ has a parity $(g,f)$-factor if and only if 
for all disjoint subsets $A$ and $B$ of $V(G)$,
$$\omega_{f}(G, A,B)<2+ \sum_{v\in A} f(v)+\sum_{v\in B} (d_{G}(v)-g(v))-d_G(A,B),$$
where $\omega_{f}(G, A,B)$ denotes the number of components $G[X]$ 
of $G\setminus (A\cup B)$
satisfying $\sum_{v\in X}f(v)\stackrel{2}{\not\equiv}d_G(X,B)$.
}\end{lem}
\subsection{Parity $\varepsilon$-factors}
The following theorem provides a parity version for Theorem~\ref{thm:Kano-Saito} on edge-connected graphs. 
Note that some edge-connected versions of that theorem were former studied in~\cite{Anstee-Nam, Egawa-Kano, Kano}.
\begin{thm}\label{thm:parity:version}
{Let $G$ be a connected graph, let $f:V(G)\rightarrow Z_2$ be a mapping with
$ \sum_{v\in V(G)}f(v)\stackrel{2}{\equiv}0$, and let $\varepsilon$ be a real number with
 $0<  \varepsilon <  1$.
If for every nonempty proper subset $X$ of $V(G)$, 
 $$d_G(X)\ge \begin{cases}
1/\varepsilon,	&\text{when $\sum_{v\in X}f(v)$ is odd};\\
1/(1-\varepsilon),	&\text{when $\sum_{v\in X}(d_G(v)-f(v))$ is odd},
\end {cases}$$
then $G$ has an $f$-parity factor $F$ such that for each vertex $v$,
 $$ \lfloor  \varepsilon d_G(v) \rfloor-1\le d_{F}(v)\le \lceil \varepsilon d_G(v) \rceil +1.$$
Furthermore, for a given arbitrary vertex $z$, we can arbitrary have $d_F(z)\ge \varepsilon d_G(z)$ or $d_F(z)\le \varepsilon d_G(z)$.
}\end{thm}
\begin{proof}{
For each vertex $v$, let us define $g_0(v)\in
 \{\lfloor \varepsilon d_G(v)\rfloor-1,
 \lfloor \varepsilon d_G(v)\rfloor\}$ 
and $f_0(v)\in \{\lceil \varepsilon d_G(v)\rceil,\lceil \varepsilon d_G(v)\rceil+1 \}$
 such that 
$g_0(v)\stackrel{2}{\equiv}f_0(v)\stackrel{2}{\equiv} f(v)$.
If $g(z)<\varepsilon d_G(z)$ and our goal is to impose that $d_F(z)\ge \varepsilon d_G(z)$,
we will replace $g_0(z)$ by $g_0(z)+2$,
and if $f(z)>\varepsilon d_G(z)$
and our goal is to impose that $d_F(z)\le \varepsilon d_G(z)$,
we will replace $f_0(z)$ by $f_0(z)-2$.
Let $A$ and $B$ be two disjoint vertex subsets of $V(G)$ with $A\cup B\neq \emptyset$. 
By the definition of $g$ and $f$, we must have 
\begin{equation}\label{eq:1:thm:epsilon:parity}
{\sum_{v\in A} \varepsilon d_G(v)+\sum_{v\in B} (1-\varepsilon) d_G(v)
< 2+ \sum_{v\in A} f_0(v)+\sum_{v\in B} (d_{G}(v)-g_0(v)),
}\end{equation}
whether $z\in A\cup B$ or not.
Take $P$ to be the collection of all vertex sets $X$ such that $G[X]$
 is a component of $G\setminus (A\cup B)$ satisfying $\sum_{v\in X}f(v)\stackrel{2}{\not\equiv} d_G(X,B)$.
It is easy to check that
\begin{equation}\label{eq:2:thm:epsilon:parity}
{\sum_{X\in P}(\varepsilon d_G(X,A)+(1-\varepsilon) d_G(X,B))
\le \sum_{v\in A} \varepsilon d_G(v)+\sum_{v\in B} (1-\varepsilon) d_G(v)\, -d_G(A,B).
}\end{equation}
Define 
$P_c=\{X\in P: d_G(X,A)>0 \text{ and }d_G(X,B)>0\}$. Obviously, 
$$|P_c|= \sum_{X\in P_c}(\varepsilon+(1-\varepsilon))\le 
\sum_{X\in P_c}(\varepsilon d_G(X,A)+(1-\varepsilon) d_G(X,B)).$$
Set $P_a=\{X\in P: d_G(X,A)=0\}$ and 
$P_b=\{X\in P: d_G(X,B)=0\}$.
According to the definition, for every $X\in P_a$, we have 
$$\sum_{v\in X}f(v)\stackrel{2}{\not\equiv} 
d_G(X,B)=d_G(X)\stackrel{2}{\equiv}\sum_{v\in X}d_G(v),$$
which implies that $\sum_{v\in X}(d_G(v)-f(v))$ is odd.
Similarly, for every $X\in P_b$, $\sum_{v\in X}f(v)$ must be odd.
Thus by the assumption, we must have
$$|P_a|\le \sum_{X\in P_a}(1-\varepsilon) d_G(X) =
 \sum_{X\in P_a}(\varepsilon d_G(X,A)+(1-\varepsilon) d_G(X,B)),$$
and also
$$|P_b|\le \sum_{X\in P_b}\varepsilon d_G(X) =
 \sum_{X\in P_b}(\varepsilon d_G(X,A)+(1-\varepsilon) d_G(X,B)).$$
Therefore, 
\begin{equation}\label{eq:3:thm:epsilon:parity}
{|P|=|P_a|+|P_b|+|P_c|
\le \sum_{X\in P}(\varepsilon d_G(X,A)+(1-\varepsilon) d_G(X,B)).
}\end{equation}
According to Relations~(\ref{eq:1:thm:epsilon:parity}), (\ref{eq:2:thm:epsilon:parity}), and (\ref{eq:3:thm:epsilon:parity}), 
one can conclude that $$\omega_{f}(G, A,B)=|P|<2+ \sum_{v\in A} f_0(v)+\sum_{v\in B} (d_{G}(v)-g_0(v))-d_G(A,B).$$
When both of the sets $A$ and $B$ are empty, the above-mentioned inequality must automatically hold, 
because $\omega_{f}(G, \emptyset,\emptyset)=0$. Thus the assertion follows from Lemma~\ref{lem:Lovasz:parity}.
}\end{proof}
\begin{remark}
{Note that if we had replaced the weaker condition $|d_F(z)-\varepsilon d_G(z)|\le 1$ for the vertex $z$, we could impose 
this condition for another vertex as well. In particular, we could impose this condition for three arbitrary vertices
 provided that $\varepsilon=1/2$.
}\end{remark}
The following corollary is an improved version of the main result in \cite{MR785652} 
by replacing an odd edge-connectivity condition. 
Note that $r$-regular bipartite graphs are in the class of odd-$r$-edge-connected graphs. 
\begin{cor}{\rm (\cite{BERMOND19849, MR785652})} 
{Let $r$ and and $r_0$ be two odd positive integers with $r\ge r_0$.
If $G$ is an odd-$\lceil r/r_0\rceil$-edge-connected $r$-regular graph, then it contains an $r_0$-factor.
}\end{cor}
\begin{proof}
{Apply Theorem~\ref{thm:parity:version} with $\varepsilon=r_0/r$ and $f(v)=r_0$ (mod $2$).
}\end{proof}
\begin{cor}{\rm (\cite{ModuloBounded})}
{Let $G$ be a graph and let $f:V(G)\rightarrow Z_2$ be a mapping with
$ \sum_{v\in V(G)}f(v)\stackrel{2}{\equiv}0$. If $G$ is $2$-edge-connected, 
then it has an $f$-parity factor $F$ such that for each vertex $v$, 
 $$ \lfloor \frac{d_G(v)}{2} \rfloor -1\le d_{F}(v)\le \lceil \frac{d_G(v)}{2} \rceil+1.$$
Furthermore, for a given arbitrary vertex $z$, we can arbitrary have
 $d_F(z)\ge d_G(z)/2$ or $d_F(z)\le  d_G(z)/2$.
}\end{cor}
\begin{proof}
{Apply Theorem~\ref{thm:parity:version} with $\varepsilon=1/2$.
}\end{proof}
\begin{cor}
{Let $G$ be a  graph and 
 let $f:V(G)\rightarrow Z_2$ be a mapping with
$ \sum_{v\in V(G)}f(v)\stackrel{2}{\equiv}0$.
If $G$ is $3$-edge-connected, then it has an $f$-parity factor $F$ such that for each vertex $v$, 
 $$ \lfloor \frac{2\, d_G(v) }{3}\rfloor -1\le d_{F}(v)\le \lceil \frac{2\, d_G(v) }{3} \rceil+1.$$
}\end{cor}
\begin{proof}
{Apply Theorem~\ref{thm:parity:version} with $\varepsilon=2/3$.
}\end{proof}
An immediate consequence of the following corollary was appeared in~\cite{Lovasz} which says that every $2$-edge-connected graph with minimum degree at least three admits a factor whose degrees are positive and even. 
\begin{cor}\label{cor:even-factor:1}
{Let $\varepsilon $ be a real number with $0\le \varepsilon\le 2/3$.
If $G$ is a $2$-edge-connected graph, then it has an even factor $F$ such that for each vertex $v$, 
 $$ | d_{F}(v)-\varepsilon d_G(v)| < 2.$$
}\end{cor}
The next corollary provides an improved version for a result in \cite{Lu2015} due to Lu, Wang, and Lin (2015) who proved that every $2m$-edge-connected graph with minimum degree at least $2m + 1$ contains an even factor with minimum degree at least $2m$.
\begin{cor}\label{cor:even-factor:2}
{Let $G$ be a $2$-edge-connected graph and let $\varepsilon $ be a real number with $2/3\le \varepsilon< 1$.
If $G$ is odd-$\lceil 1/(1-\varepsilon)\rceil $-edge-connected, then it has an even factor $F$ such that for each vertex~$v$, 
 $$ | d_{F}(v)-\varepsilon d_G(v)|<2.$$
}\end{cor}
When the main graph has no odd edge-cuts, one can derive the following simpler version of Corollaries~\ref{cor:even-factor:1} and~\ref{cor:even-factor:2}.
\begin{cor}\label{cor:Eulerian}
{Let $\varepsilon $ be a real number with $0\le  \varepsilon\le 1$.
If $G$ is a graph with even degrees, then it admits an even factor $F$ such that for each vertex $v$, 
 $$ | d_{F}(v)-\varepsilon d_G(v)| < 2.$$
}\end{cor}
Here, we introduce  a simple inductive proof for Theorem 8 in~\cite{{Hilton1982}} based on Corollary~\ref{cor:Eulerian}.
\begin{cor}{\rm (\cite{Hilton1982})}\label{thm:Hilton}
{Every  graph $G$ with even degrees can be edge-decomposed into $k$ even factors $G_1,\ldots, G_k$
 such that for each vertex $v$, 
 $|d_{G_i}(v)-d_{G}(v)/k|<2.$
}\end{cor}
\begin{proof}
{By induction on $k$. We may assume that $k\ge 2$ as the assertion trivially holds when  $k=1$. According to Corollary~\ref{cor:Eulerian}, the graph can be  edge-decomposed into two even factors $G'$ and $G_{k}$  such that for each vertex $v$, 
 $|d_{G_k}(v)-\frac{1}{k}d_{G}(v)|<2$. By the induction hypothesis, the graph $G'$ can be edge-decomposed into 
$k-1$ even factors $G_1,\ldots, G_{k-1}$
 such that for each vertex $v$, 
 $|d_{G_i}(v)-\frac{1}{k-1}d_{G'}(v)|<2.$
We claim that these are the desired factors we are looking for.
Let $v\in V(G)$.
 Since $d_{G'}(v)\le  \frac{k-1}{k}d_{G}(v)+2-\frac{1}{k}$, we must have 
$d_{G_i}(v) \le \frac{1}{k-1}(d_{G'}(v)-1)+2\le \frac{1}{k}(d_G(v)+1)+2.$
Since $d_G(v)$ is even, 
$(d_G(v)+1)/k$ must not be  an even integer number, which implies that 
$d_{G_i}(v)  \le \frac{1}{k}d_G(v)+2$.
Since $d_{G'}(v)$ is even, if $\frac{1}{k}d_{G}(v)$ is an even integer number, then $d_{G'}(v)\le  \frac{k-1}{k}d_{G}(v)$. 
Therefore, since $d_{G_i}(v)$ is even, one can conclude that $d_{G_i}(v) < \frac{1}{k}d_G(v)+2$ 
whether $\frac{1}{k}d_{G}(v)$  is an integer number or not. 
Similarly, we  have $d_{G_i}(v) > \frac{1}{k}d_{G}(v)-2$. Hence the assertion holds.
}\end{proof}
By replacing Corollary~\ref{cor:Eulerian} in the proof of 
 Theorem~\ref{thm:Correa-Goemans}, one can formulate  a parity version of it as the following theorem. It would be interesting to determine the sharp upper bound on vertex degrees to make another generalization for Corollary~\ref{cor:Eulerian}.
\begin{thm}
{Let $\varepsilon_1,\ldots,\varepsilon_k$ be $k$ nonnegative real numbers with $\varepsilon_1+\cdots +\varepsilon_k=1$.
If $G$ is a graph with even degrees, then it can be edge-decomposed into $k$ even factors $G_1,\ldots, G_k$ such that for each $v\in V(G_i)$,
 $$|d_{G_i}(v)-\varepsilon_i d_G(v)|<6.$$
}\end{thm}
\begin{proof}
{Apply Corollary~\ref{cor:Eulerian} along with the same arguments stated in the proof of Theorem 4 in \cite{Correa-Goemans}.
}\end{proof}
%
%
%
%
%
%
%
%
%

\end{document}